\documentclass[12pt,a4paper,reqno]{amsart}
\usepackage{geometry}
\usepackage{amsmath,amssymb}
\usepackage{amsfonts}
\usepackage{eucal}
\usepackage{amsthm}
\usepackage{mathrsfs}
\numberwithin{equation}{section}
\usepackage{comment}
\newcommand{\jap}[1]{\langle #1 \rangle}

\def\a{\alpha}

\def\c{\gamma}

\def\e{\varepsilon}

\def\g{\psi}

\def\m{\mu}

\def\re{\mathbb{R}}

\def\pa{\partial}

\newtheorem{thm}{Theorem}[section]
\newtheorem{lem}[thm]{Lemma}
\newtheorem{prop}[thm]{Proposition}
\newtheorem{cor}[thm]{Corollary}

\theoremstyle{definition}

\theoremstyle{remark}
\newtheorem{rem}[thm]{Remark}

\title[]%
{A remark on Strichartz estimates for Schr\"odinger equations with slowly decaying potentials}

\author{Kouichi Taira}
\address{Department of Mathematical Sciences, Ritsumeikan University, 1-1-1 NojiHigashi, Kusatsu, 525-8577 Japan}
\email{ktaira@fc.ritsumei.ac.jp}
\date{}

\begin{document}
\maketitle

\begin{abstract}
In this short note, we prove Strichartz estimates for Schr\"odinger operators with slowly decaying singular potentials in dimension two. This is a generalization of the recent results by Mizutani, which are stated for dimension greater than two. The main ingredient of the proof is a variant of Kato's smoothing estimate with a singular weight.
\end{abstract}


\section{Introduction}

The purpose of this note is to prove the Strichartz estimates for Schr\"odinger equations with  slowly decaying singular potentials in dimension two. Such results have been recently proved by Mizutani \cite{M2} in dimension greater than two. Throughout this paper, we assume $0<\m<2$. We consider the two-dimensional admissible condition:
\begin{align}\label{admissible}
p\geq 2,\quad \frac{2}{p}+\frac{2}{q}=1,\quad (p,q)\neq (2,\infty).
\end{align}
The main result of this note is the following theorem:

\begin{thm}\label{twoSt}
Let $n=2$, $Z>0$ and $H_1=-\Delta+Z|x|^{-\m}+\e V_S(x)$, where $V_S\in C^{\infty}(\re^2;\re)$ satisfies
\begin{align*}
|\pa_x^{\a}V_S(x)|\leq C_{\a}(1+|x|)^{-1-\m-|\a|}
\end{align*}
for all $\a\in\mathbb{Z}_+^2$.
Then there exists $\e_*>0$ such that for all $\e\in [0,\e_*)$ and $(p,q), (\tilde{p},\tilde{q})$ satisfying $(\ref{admissible})$, there exists $C>0$ such that
\begin{align}
\|e^{-itH_1}u_0\|_{L^p(\re;L^q(\re^2))}\leq& C\|u_0\|_{L^2(\re^2)},\label{hom}\\
\|\int_0^te^{-i(t-s)H_1}F(s)ds\|_{L^p(\re;L^q(\re^2))}\leq& C\|F\|_{L^{\tilde{p}^*}(\re;L^{\tilde{q}^*}(\re^2))}\label{Inhom}
\end{align}
for all $u_0\in L^2(\re^2)$ and $F\in L^{\tilde{p}^*}(\re;L^{\tilde{q}^*}(\re^2))$.

\end{thm}

Moreover, we introduce a family of smooth slowly decaying potential with repulsive conditions.
Assume that a real-valued function $V$ satisfies

\noindent $(i)$ $|\pa_x^{\a}V(x)|\leq C_{\a}(1+|x|)^{-\m-|\a|}$

\noindent $(ii)$ There exists $C>0$ such that $V(x)\geq C(1+|x|)^{-\m}$.

\noindent $(iii)$ There exist $C>0$ and $R>0$ such that $-x\cdot \pa_xV(x)\geq C|x|^{-\m}$ for $|x|>R$.

\noindent Moreover, we set 
\begin{align*}
H=H_0+V,\quad H_0=-\Delta.
\end{align*}

As is stated above, Theorem \ref{twoSt} is proved in \cite[Theorem 1.2]{M2} for dimension $n\geq 3$. Its proof is based on the Strichartz estimates for the propagator $e^{-itH}$ and the smooth perturbation method developed in \cite{BM} and \cite{RodSc}.
To control the local singularities $|x|^{-\m}$, we need Kato's smoothing estimates with the weight $\chi(x)|x|^{-\m}$ with $\chi\in C_c^{\infty}(\re^n)$.
For $n\geq 3$, such estimates follow from the endpoint Strichrtz estimates for $e^{-itH}$ and the result in \cite{BM} (see \cite[Remark 1.6]{M2} and the proof of \cite[Theorem 1.2]{M2}). However, since the endpoint Strichartz estimates might not be hold in dimension two, we cannot use this strategy for $n=2$. In this note, we supply such a space-time $L^2$-estimate stated in \cite[Remark 1.6]{M2} for $n=2$. More strongly, we have the following theorem.

\begin{thm}\label{main}
Suppose $n=2$ and let $\chi\in C_c^{\infty}(\re^2)$. Then we have
\begin{align*}
\sup_{z\in\mathbb{C}\setminus \re}\|\chi(x)|x|^{-\frac{\m}{2}}(H-z)^{-1}|x|^{-\frac{\m}{2}}\chi(x)\|<\infty.
\end{align*}
\end{thm}

\begin{rem}
Using the method written in the next section, we can prove a local smoothing estimate:
\begin{align}\label{locsmooth}
\sup_{z\in\mathbb{C}\setminus \re}\|\jap{x}^{-\c}\jap{D_x}^{\frac{1}{2}}(H-z)^{-1}\jap{D_x}^{\frac{1}{2}} \jap{x}^{-\c}\|<\infty,
\end{align}
for $\c>\frac{1}{2}+\frac{\m}{4}$ and for $n\geq 1$, that is, $\jap{x}^{-\c}\jap{D_x}^{\frac{1}{2}}$ is $H$-supersmooth. In fact, $(\ref{locsmooth})$ follows from $(\ref{semicl})$ and $(\ref{generalbd1})$, which are true for each $n\geq 1$. We also recall that $\jap{x}^{-1}\jap{D_x}^{\frac{1}{2}}$ is $H_0$-supersmooth for $n\geq 3$ (\cite{KY}).
\end{rem}

By the smooth perturbation theory \cite[Theorem XIII.25]{RS}, we obtain a global in time estimate for the propagator $e^{-itH}$.

\begin{cor}\cite[Conjectured in Remark 1.6]{M2}\label{corsmooth}
Suppose $n=2$ and let $\chi\in C_c^{\infty}(\re^2)$. Then we have
\begin{align*}
\|\chi(x)|x|^{-\frac{\m}{2}}e^{-itH}u_0\|_{L^2(\re^{3})}\leq C\|u_0\|_{L^2(\re^2)}\quad u_0\in L^2(\re^2).
\end{align*}
\end{cor}

\begin{rem}
By $(\ref{locsmooth})$, we have $\|\jap{x}^{-\c}\jap{D_x}^{\frac{1}{2}}e^{-itH}u_0\|_{L^2(\re^{n+1})}\leq C\|u_0\|_{L^2(\re^n)}$ for $u_0\in L^2(\re^n)$, $n\geq 1$ and $\c>\frac{1}{2}+\frac{\m}{4}$.

\end{rem}

\begin{proof}[Proof of Theorem \ref{twoSt} assuming Corollary \ref{corsmooth}] The inhomogeneous estimates  $(\ref{Inhom})$ follow from the homogeneous estimates $(\ref{hom})$ and the Christ-Kiselev lemma \cite{CK}. Thus, we only need to prove the homogeneous estimates $(\ref{hom})$. Let $\chi\in C_c^{\infty}(\re^n;[0,1])$ such that $\chi(x)=1$ on the unit ball.
Set $W=Z|x|^{-\m}+\e V_S$. Following the proof of \cite[Theorem 1.2]{M2}, write $W(x)=V_1(x)+V_2(x)$, where
\begin{align*}
V_1(x)=\chi(x)^2+(1-\chi(x)^2)W(x),\quad V_2(x)=\chi(x)^2(W(x)-1).
\end{align*}
Then it turns out that $V_1$ satisfies the condition $(i)$, $(ii)$ and $(iii)$. Set $B=|V_2|^{\frac{1}{2}}$ and $A=|V_2|^{\frac{1}{2}}\mathrm{sgn} V_2$. By virtue of the Strichartz estimates for $e^{-it(H_0+V_1)}$ (\cite[Theorem 1.1]{M2}) and \cite[Theorem 4.1]{RodSc}, then it suffices for $(\ref{hom})$ to prove that $B$ is $(H_0+V_1)$-smooth and $A$ is $(H_0+W$)-smooth. Since $|A|, |B|\leq C\chi(x)|x|^{-\frac{\m}{2}}$ with a constant $C>0$, we only need to prove that $\chi(x)|x|^{-\frac{\m}{2}}$ is both $(H_0+V_1)$-smooth and $(H_0+W$)-smooth. The former follows from Corollary \ref{corsmooth} and the latter follows from \cite[Proposition 5.1]{M2} with $\e>0$ small enough, where we note that \cite[Proposition 5.1]{M2} holds for $n=2$ since its proof is based on \cite[Corollary 2.21]{BM}. This completes the proof.
\end{proof}

In the rest of this paper, we prove Theorem \ref{main}.

We use the following notations throughout this paper. For Banach spaces $X$ and $Y$, $B(X,Y)$ denotes a set of all bounded linear operators from $X$ to $Y$. We denote the norm of a Banach space $X$ by $\|\cdot\|_X$. We write $\jap{x}=(1+|x|^2)^{\frac{1}{2}}$ and $D_x=(2\pi i)^{-1}\nabla_x$. 
We also denote $\|\cdot\|_{p\to q}=\|\cdot\|_{B(L^p,L^q)}$ and $\|\cdot\|=\|\cdot \|_{2\to 2}$.

\noindent\textbf{Acknowledgment.}  
 The author would like to thank Haruya Mizutani for suggesting this problem.

\section{Proof of Theorem \ref{main}}

\subsection{Preliminary}

\begin{lem}
Let $\g\in C_c^{\infty}(\re)$ and $s, k\in \re$.  Then we have
\begin{align}\label{Bound1}
\chi(x)|x|^{-\frac{\m}{2}}\jap{D_x}^{-\frac{\m}{2}}\jap{x}\in B(L^2(\re^2)),\quad \jap{x}^{-k}\jap{D_x}^{s}\g(H)\jap{x}^k\in B(L^2(\re^2)).
\end{align}
\end{lem}

\begin{proof}
By Hardy's inequality, we have $|x|^{-\frac{\m}{2}}\jap{D_x}^{-\frac{\m}{2}}\in B(L^2(\re^2))$, where we note $\frac{\m}{2}<1=\frac{n}{2}$ for $n=2$. Moreover, a simple commutator argument gives $\jap{D_x}^{\frac{\m}{2}}\chi(x)\jap{D_x}^{-\frac{\m}{2}}\jap{x}\in B(L^2(\re^2))$.
Writing $\chi(x)|x|^{-\frac{\m}{2}}\jap{D_x}^{-\frac{\m}{2}}\jap{x}=|x|^{-\frac{\m}{2}}\jap{D_x}^{-\frac{\m}{2}}\cdot \jap{D_x}^{\frac{\m}{2}}\chi(x)\jap{D_x}^{-\frac{\m}{2}}\jap{x}$, we obtain $\chi(x)|x|^{-\frac{\m}{2}}\jap{D_x}^{-\frac{\m}{2}}\jap{x}\in B(L^2(\re^2))$. The boundedness of $\jap{x}^{-1}\jap{D_x}^{\frac{\m}{2}}\g(H)\jap{x}$ immediately follows from a simple commutator argument.

\end{proof}

\begin{lem}
We have
\begin{align}\label{semicl}
\sup_{z\in\mathbb{C}\setminus \re,\,\,|z|\geq 1}\|\jap{x}^{-\c}\jap{D_x}^{s}(H-z)^{-1}\jap{D_x}^{s} \jap{x}^{-\c}\|<\infty,
\end{align}
for $0\leq s\leq \frac{1}{2}$ and $\c>\frac{1}{2}$.
\end{lem}

\begin{proof}
This lemma directly follows from a high-energy resolvent estimates due to the assumption on the potential $(i)$.
\end{proof}

\begin{prop}
We have
\begin{align}\label{generalbd1}
\sup_{z\in\mathbb{C}\setminus \re,\,\,|z|\leq 1}\|\jap{x}^{-\c}\jap{D_x}^s(H-z)^{-1}\jap{D_x}^s \jap{x}^{-\c}\|<\infty,
\end{align}
for $0\leq s\leq 1$ and $\c>\frac{1}{2}+\frac{\m}{4}$.
\end{prop}

\begin{rem}
In our proof of Theorem \ref{main}, the repulsive condition $(ii)$ and $(iii)$ are used for this proposition only.
\end{rem}

\begin{proof}
 Let $\g\in C_c^{\infty}(\re)$ satisfying $\g(t)=1$ on $|t|\leq 2$.
By $(\ref{Bound1})$ and Nakamura's result (\cite[Theorem 1.8]{N}):
\begin{align*}
\sup_{z\in \mathbb{C}\setminus \re}\|\jap{x}^{-\c}(H-z)^{-1}\jap{x}^{-\c}\|<\infty,
\end{align*}
it turns out that
\begin{align}\label{generalbd2}
\sup_{z\in\mathbb{C}\setminus \re}\|\jap{x}^{-\c}\jap{D_x}^{s}\g(H)(H-z)^{-\c}\g(H) \jap{D_x}^{s} \jap{x}^{-1}\|<\infty.
\end{align}
We note that $(H-z)^{-1}(1-\g(H)^2)$ is bounded from $H^{s}(\re^2)$ to $H^{-s}(\re^2)$ and its operator norm is uniformly bounded in $|z|\leq 1$. Thus $(\ref{generalbd1})$ follows from the estimate $(\ref{generalbd2})$.
\end{proof}

\subsection{The case $0<\m\leq 1$}

In this subsection, we assume $0<\m\leq 1$. By virtue of $(\ref{Bound1})$, for proving Theorem \ref{main}, we only need to show
\begin{align}\label{Bound2}
\sup_{z\in \mathbb{C}\setminus \re}\|\jap{x}^{-1}\jap{D_x}^{\frac{\m}{2}}(H-z)^{-1}\jap{D_x}^{\frac{\m}{2}} \jap{x}^{-1}\|<\infty.
\end{align}
This bound directly follows from $(\ref{semicl})$ and $(\ref{generalbd1})$ since $\frac{\m}{2}\leq \frac{1}{2}$.
This completes the proof.

\subsection{The case $1<\m< 2$}

In this subsection, we assume $1<\m< 2$. In this case, the inequality $(\ref{Bound2})$ might be false even if $H$ is replaced by $H_0$, where the problem is high energy case $|z|\geq 1$.

 By $(\ref{Bound1})$ and $(\ref{generalbd1})$, for proving Theorem \ref{main}, it suffices to show that
\begin{align}\label{Bound3}
\sup_{z\in \mathbb{C}\setminus \re,\,|z|\geq 1}\|\chi(x)|x|^{-\frac{\m}{2}}(H-z)^{-1}|x|^{-\frac{\m}{2}}\chi(x)\|<\infty.
\end{align}
We denote $R(z)=(H-z)^{-1}$ and $R_0(z)=(H_0-z)^{-1}$. We shall prove more general estimates
\begin{align}\label{Bound4}
\sup_{z\in \mathbb{C}\setminus \re,\,|z|\geq 1}\|R(z)\|_{p\to p^*}<\infty,\quad \text{for}\quad 1<p\leq \frac{6}{5},
\end{align}
which are proved in \cite{IS} for middle energy case $|z|\sim 1$.
In fact, since $\chi(x)|x|^{-\frac{\m}{2}}\in L^q(\re^2)$ for some $2<q\leq 3$, the estimate $(\ref{Bound3})$ follows from H\"older's inequality and $(\ref{Bound4})$.

By the resolvent identity, we have
\begin{align*}
R(z)=R_0(z)-R_0(z)VR_0(z)+R_0(z)VR(z)VR_0(z),
\end{align*}
which implies
\begin{align}
\|R(z)\|_{p\to p^*}\leq& \|R_0(z)\|_{p\to p^*}+\|R_0(z)\|_{B(\mathcal{B},L^{p^*})}\|V\|_{B(\mathcal{B}^*, \mathcal{B})}\|R_0(z)\|_{B(L^p,\mathcal{B}^*)}\nonumber\\
&+\|R_0(z)\|_{B(\mathcal{B},L^{p^*})}\|R(z)\|_{B(\mathcal{B}, \mathcal{B}^*)}\|V\|_{B(\mathcal{B}^*, \mathcal{B})}^2\|R_0(z)\|_{B(L^p,\mathcal{B}^*)},\label{Bound5}
\end{align}
where we denote $D_j=\{x\in \re^2\mid |x|\in [2^{j-1},2^j]\}$ for $j\geq 1$, $D_0=\{x\in \re^2\mid |x|\leq 1\}$ and
\begin{align*}
\mathcal{B}=\{u\in L^2_{loc}(\re^2) \mid \sum_{j=0}^{\infty}2^{j/2}\|u\|_{L^2(D_j)}<\infty\},\,\, \mathcal{B}^*=\{u\in L^2_{loc}(\re^{2})\mid \sup_{j\geq 0}2^{-j/2}\|u\|_{L^2(D_j)}<\infty\}.
\end{align*}
Since $\jap{x}^{-s}L^2(\re^2)\subset \mathcal{B}\subset \mathcal{B}^*\subset \jap{x}^{s}L^2(\re^2)$ for $s>\frac{1}{2}$ and $|V(x)|\leq C\jap{x}^{-\m}$ with $\m>1$, we have $V\in B(\mathcal{B}^*, \mathcal{B})$. The bounds for the free operator
\begin{align*}
\sup_{z\in \mathbb{C}\setminus \re,\,|z|\geq 1}\|R_0(z)\|_{p\to p^*},\,\,  \sup_{z\in \mathbb{C}\setminus \re,\,|z|\geq 1}\|R_0(z)\|_{B(\mathcal{B},L^{p^*})},\,\,\sup_{z\in \mathbb{C}\setminus \re,\,|z|\geq 1}\|R_0(z)\|_{B(L^p,\mathcal{B}^*)}<\infty
\end{align*}
are well-known (\cite{IS}, \cite{KRS}, \cite{RV}). Moreover, the standard limiting absorption principle between Besov spaces gives
\begin{align*}
\sup_{z\in \mathbb{C}\setminus \re,\,|z|\geq 1}\|R(z)\|_{B(\mathcal{B}, \mathcal{B}^*)}<\infty.
\end{align*}
Hence the right hand side of $(\ref{Bound5})$ is uniformly bounded in $|z|\geq 1$, which implies $(\ref{Bound4})$.

\begin{rem}
In the proof above, we can avoid the use of Besov spaces. In fact, if we take $s>\frac{1}{2}$ such that $V\in B(\jap{x}^{s}L^2(\re^2), \jap{x}^{-s}L^2(\re^2))$ and if replace $\mathcal{B}$ by $\jap{x}^{-s}L^2(\re^2)$ and $\mathcal{B}^*$ by $\jap{x}^{s}L^2(\re^2)$, the proceeding argument remains true.
\end{rem}

\end{document}